\documentclass[10pt]{amsart}

\usepackage{eucal,fullpage,times,amsmath,amsthm,amssymb,mathrsfs,stmaryrd,color,enumerate,accents,enumitem}
\usepackage[all]{xy}
\usepackage{url}

\usepackage[colorlinks=true,hyperindex, citecolor=cyan, urlcolor=cyan, pagebackref=true]{hyperref}

\begin{document}

\bibliographystyle{alpha}

\newtheorem{theorem}{Theorem}[section]
\newtheorem{proposition}[theorem]{Proposition}
\newtheorem{lemma}[theorem]{Lemma}
\newtheorem{corollary}[theorem]{Corollary}

\theoremstyle{definition}
\newtheorem{definition}[theorem]{Definition}
\newtheorem{remark}[theorem]{Remark}
\newtheorem{example}[theorem]{Example}
\newtheorem{notation}[theorem]{Notation}
\newtheorem{construction}[theorem]{Construction}

\title{Vanishing theorems for perverse sheaves on abelian varieties, revisited}

\begin{abstract} 
We revisit some of the basic results of generic vanishing theory, as pioneered by Green and Lazarsfeld, in the context of constructible sheaves. Using the language of perverse sheaves, we give new proofs of some of the basic results of this theory. Our approach is topological/arithmetic, and avoids Hodge theory.
\end{abstract}
\author{Bhargav Bhatt}
\author{Christian Schnell}
\author{Peter Scholze}

\maketitle

\section{Introduction}

Let $A$ be a compact complex torus of dimension $g$. Fix a field $k$ of coefficients, and let
\[ \mathrm{Char}(A) := \underline{\mathrm{Hom}} \big( \pi_1(A), \mathbb G_{m,k}\big)\simeq \mathbb G_{m,k}^{2g} \]
be the $k$-linear character variety of the fundamental group; this is a torus of dimension $2g$ over $k$. Given a character $\chi \in \mathrm{Char}(A)$, we write $L_\chi$ for the associated rank-one local system on $A$. The goal of this paper is to revisit some results describing the behaviour of the cohomology $M \otimes L_\chi$, where $M$ is a fixed constructible sheaf (or complex) on $A$, and $\chi$ varies through points of $\mathrm{Char}(A)$. These results constitute the constructible sheaf variants of the pioneering work of Green and Lazarsfeld \cite{GreenLazarsfeld, GreenLazarsfeldHigherObs}, and have been revisited by many authors in the interim.

\subsection{Generic vanishing}

The first goal of this paper is to give a new short proof for the following vanishing
theorem, which is new in this generality.

\begin{theorem} \label{thm:GV}
Let $M \in \mathrm{Perv}(A, k)$ be a perverse sheaf with $k$-coefficients. Then
\[
	H^i \big( A, M \otimes_k L_\chi \big) = 0 \quad \text{for all $i \neq 0$}
\]
for $\chi$ lying in a non-empty Zariski open subset of $\mathrm{Char}(A)$.
\end{theorem}

By a standard argument, Theorem~\ref{thm:GV} can be reformulated as the following assertion about constructible sheaves: 

\begin{corollary} \label{cor:GV}
Let $F$ be a constructible sheaf on $A$ with $k$-coefficients.  Then
\[
	H^i \big( A, F \otimes_k L_\chi \big) = 0 \quad \text{for all $i  > \dim(\mathrm{Supp}\ F)$}
\]
for $\chi$ lying in a non-empty Zariski open subset of $\mathrm{Char}(A)$.
\end{corollary}

Theorem~\ref{thm:GV} is an analogue for perverse sheaves of the ``generic vanishing theorem'' of Green and Lazarsfeld \cite{GreenLazarsfeld}; indeed, Theorem~\ref{thm:GV} implies the Kodaira-Nakano type vanishing results of \cite{GreenLazarsfeld} via Hodge theory. 

When $A$ is an abelian variety and $k = \mathbf{C}$, Theorem~\ref{thm:GV} was first proven in \cite{Kraemer+Weissauer:Vanishing} using hard Lefschetz and Tannakian categories arising via convolution, and independently in \cite{Schnell:holonomic} via Laumon-Rothstein's Fourier transform for $\mathcal{D}$-modules on $A$. The new idea of our proof is to pass to the universal covering space of $A$, which is a complex
vector space, and then to apply Artin's vanishing theorem for perverse sheaves on Stein manifolds; this was inspired by \cite[\S IV]{Scholze:Torsion}, which proves a vanishing theorem for the $\mathbf{F}_p$-cohomology of Shimura varieties by invoking the analog of Artin vanishing for the ``perfectoid universal cover'' of the Shimura variety. The implementation of this idea relies on the Fourier-Mellin transform for constructible sheaves on abelian varieties, which coherently interpolates the cohomology of all character twists $M\otimes L_\chi$ of $M$. More precisely, the Fourier-Mellin transform is given by a functor 
\[ \mathrm{FM}_A:D^b_c(A,k) \to D^b_{coh}(\mathrm{Char}(A))\]
from the constructible derived category of $k$-linear sheaves on $A$ to coherent complexes on $\mathrm{Char}(A)$, such that the fiber of $\mathrm{FM}_A(M)$ at a point $\chi \in \mathrm{Char}(A)$ is the cohomology of $M \otimes L_\chi$.

\subsection{Codimension estimates}
From Theorem~\ref{thm:GV} and some basic properties of the Fourier transform, one can deduce some additional results about the ``cohomology support loci''
\[ S^i(A, M) := \{\chi \in \mathrm{Char}(A) \mid H^i( A, M \otimes_k L_\chi) \neq 0\}\ ,\]
which are easily seen to be Zariski closed subsets of $\mathrm{Char}(A)$. The proof of Theorem~\ref{thm:GV} immediately yields that $\mathrm{codim}(S^i(A,M)) \geq |i|$ for all $i$ when $M$ is perverse. To get better estimates, we need to specialize our assumptions on $k$ and $A$.

\begin{theorem} \label{thm:t-structure}
Assume that $k$ is a field of characteristic $0$, and that $A$ is an abelian variety. Let $M \in D^b_c(A,k)$ be an algebraically constructible complex. Then one has
\begin{align*}
	M \in {}^p D^{\leq 0}_c(A,k) \quad &\Longrightarrow \quad \text{$\mathrm{codim}(S^i(A, M)) \geq 2 i$ for every $i \in \mathbf{Z}$,} \\
	M \in {}^p D^{\geq 0}_c(A,k) \quad &\Longrightarrow \quad \text{$\mathrm{codim}(S^i(A, M)) \geq -2i$ for every $i \in \mathbf{Z}$.}
\end{align*}
In particular, if $M$ is a perverse sheaf, then $\mathrm{codim}(S^i(A, M)) \geq |2i|$ for every $i \in \mathbf{Z}$.
\end{theorem}

The converse to Theorem~\ref{thm:t-structure} is also true. In fact, both results can be deduced via the Riemann-Hilbert correspondence from \cite{Schnell:holonomic}; the latter relies on Simpson's work \cite{SimpsonSubspaces} analyzing when subsets of a moduli space of rank $1$ local systems are algebraic simultaneously in the Betti, de Rham, and Dolbeault realizations. The novelty of the approach taken here to proving Theorem~\ref{thm:t-structure} is that it is essentially topological: we deduce the improved codimension estimate formally from the Hard Lefschetz theorem and properties of the Fourier transform, without ever leaving the world of constructible sheaves.

\subsection{Linearity}

The final result we shall discuss is one that gives a precise local description of the Fourier transform $\mathrm{FM}_A(M)$ for a simple perverse sheaf $M$ of geometric origin on $A$. Roughly speaking, the result states that the stalk of $\mathrm{FM}_A(M)$ at the trivial character is given by a ``linear complex'' or a ``derivative complex". The precise statement is given in Theorem~\ref{thm:Linearity} below. For this, we denote by $S$ the co-ordinate ring of the formal completion of $\mathrm{Char}(A)$ at the trivial character $1 \in \mathrm{Char}(A)$. As $\mathrm{Char}(A) = \mathrm{Spec}(k[\pi_1(A)])$, we have $S \simeq \widehat{\mathrm{Sym}}(H_1(A,k))$. 

\begin{theorem}
\label{thm:LinearityIntro}
Let $A$ be an abelian variety, and let $k$ be a field of characteristic $0$. Let $M$ be a simple $k$-linear perverse sheaf on $A$ of geometric origin. Then the completed stalk at $1$ of $\mathrm{FM}_A(M) \in D^b_{coh}(\mathrm{Char}(A))$ is represented by the $S$-complex
\[ \cdots \to H^{i-1}(A,M) \otimes_k S \to H^i(A,M) \otimes_k S \to H^{i+1}(A,M) \otimes_k S \to \cdots, \]
where the differential arises from the natural map $H^i(A,M) \to H^{i+1}(A,M) \otimes_k H_1(A,k)$ that is adjoint to the cup product  $H^1(A,k) \otimes_k H^i(A,M) \to H^{i+1}(A,M)$.
\end{theorem}

Theorem~\ref{thm:LinearityIntro} is the constructible sheaf version of the key result of \cite{GreenLazarsfeldHigherObs} (which dealt with the coherent context), and is implied by the linearity result in \cite{PopaSchnellGVMHM} (which applies in the context of mixed Hodge modules); these results form the essential ingredient in proving linearity properties of the cohomological support loci $S^i(A,M)$ of $M$. The proofs of these results in both \cite{GreenLazarsfeldHigherObs} and \cite{PopaSchnellGVMHM} rely on Hodge theory. In contrast, our proof is essentially arithmetic: we deduce Theorem~\ref{thm:LinearityIntro} by specializing to characteristic $p$, and using the theory of weights to prove a formality result that implies the theorem via a version of the BGG correspondence.

\subsection*{Acknowledgements}
We would like to thank Jacob Lurie for a very helpful conversation about \S
\ref{ss:weights}. During the preparation of this work, Bhatt was partially supported
by the NSF Grant DMS \#1501461 and a Packard fellowship, Schnell was partially
supported  by NSF grant DMS\#1404947 and by a Centennial Fellowship from the AMS, and
Scholze was supported by a Clay research fellowship. Both Bhatt and Scholze would
like to thank the University of California, Berkeley, the MSRI, and the Clay
foundation for their hospitality and support during the initiation of this project.
Both Bhatt and Schnell would like to thank the Simons Center for Geometry and Physics
for their hospitality during the program ``Complex, $p$-adic, and logarithmic Hodge
theory and their applications''.

\section{Fourier-Mellin transforms and generic vanishing}
\label{sec:FMTori}

In this section, we prove Theorem~\ref{thm:GV}. We use the formalism of constructible complexes and perverse sheaves in the setting of complex manifolds; more details can be found in \cite[Section~4.5]{Hotta+Takeuchi+Tanisaki:Book} and \cite[Chapter~4]{Dimca:Sheaves}. Many of the basic compatibility results proven below concerning the Fourier-Mellin transforms on compact complex tori are analogs of analogous results for $\ell$-adic sheaves on algebraic tori proven by Gabber-Loeser \cite{GabberLoeser}. We fix the following notation: 

\begin{notation}
Fix a field $k$, and a complex torus $A$ of dimension $g$. Write $D^b_c(A,k)$ for the bounded derived category of constructible complexes of $k$-modules on $A$; this triangulated category comes equipped with the constant sheaf abusively denoted by $k$, and Verdier's duality functor $D_{A,k}(-) := \underline{\mathrm{RHom}}_k(-, k[2g])$. 

Let $R = k[\pi_1(A)]$ be the group algebra of the fundamental group $\pi_1(A)$ of $A$; we always choose the base point at $0$. Note that $R$ is a regular noetherian ring, and $\mathrm{Char}(A) := \mathrm{Spec}(R)$ is the character variety of $A$ (relative to $k$). As before, one has the corresponding constructible derived category $D^b_c(A,R)$ of $R$-modules on $A$ and its Verdier duality functor $D_{A,R}(-) := \underline{\mathrm{RHom}}_R(-,R[2g])$. 

Let $\pi:V \to A$ be the universal cover of $A$, so $V$ is a vector space, and there is a natural action of $\pi_1(A)$ on $V$ by deck transformations. In particular, the sheaf $\mathcal{L}_R := \pi_! k$ is naturally a sheaf of $R$-modules on $A$. 
\end{notation}

The sheaf $\mathcal{L}_R$ introduced above is a central player in this work. It may be viewed as the ``universal'' rank $1$ local system on $A$ in the following sense:
 
 \begin{lemma}
The $R$-module $\mathcal{L}_R$ is locally free of rank $1$. In fact, it is the $R$-local system associated to the tautological character $can:\pi_1(A) \to R^*$.
 \end{lemma}
 \begin{proof}
Let $U \subset A$ be a  simply connected open subset. Then $\tilde{A} \times_A U \to U$ is isomorphic to $\sqcup_{\pi_1(A)} U \to U$. By proper base change, it follows that $(\pi_! k)|_U$ is identified with the constant sheaf $R$ as $R$-modules, proving the first part. The second part follows by unwinding definitions.
 \end{proof}

As $\mathcal{L}_R$ is a local system, we may dualize it to obtain another local system $\mathcal{L}^\vee_R := \underline{\mathrm{RHom}}_R(\mathcal{L}_R, R)$ of $R$-modules on $A$. One then has:

\begin{lemma}
\label{lem:DualTautological}
There is a canonical identification $[-1]^* \mathcal{L}_R \simeq \mathcal{L}_R^\vee$.
\end{lemma}
\begin{proof}
By the previous lemma, $\mathcal{L}_R$ is the local system associated to the tautological character $can:\pi_1(A) \to R^*$. Thus, its dual $\mathcal{L}_R^\vee$ corresponds to the character 
\[ \pi_1(A) \xrightarrow{can} R^* \xrightarrow{\iota} R^*,\]
where $\iota(g) = g^{-1}$. Since $[-1]_*$ acts via $-1$ on $\pi_1(A)$, the previous composition is identified with
\[ \pi_1(A) \xrightarrow{[-1]_*} \pi_1(A) \xrightarrow{can} R^*,\]
which proves the claim.
\end{proof}

This construction satisfies the following  compatibility:

\begin{lemma}
\label{lem:DualityTautological}
The following diagram is canonically commutative:
\[ \xymatrix{ D^b_c(A,k) \ar[rr]^-{\mathcal{L}_R \otimes_k (-)} \ar[d]_-{D_{A,k}(-)} & & D^b_c(A,R) \ar[d]^-{D_{A,R}(-)} \\
		  D^b_c(A,k) \ar[rr]^-{\mathcal{L}^\vee_R \otimes_k (-)} && D^b_c(A,R). }\]
Here the vertical maps are antiequivalences.
\end{lemma}
\begin{proof}
Fix some $K \in D^b_c(A,k)$. We must show that the natural map
\[ \mathcal{L}_R^\vee \otimes_k \underline{\mathrm{RHom}}_k(K,k[2g]) \to   \underline{\mathrm{RHom}}_R(\mathcal{L}_R \otimes_k K, R[2g]) \]
is an isomorphism of sheaves. This assertion is local on $A$. Moreover, since $\mathcal{L}_R$ is locally constant, it reduces to the following statement: the functor $D^b_c(A,k) \to D^b_c(A,R)$ given by extension of scalars along $k \to R$ carries $\underline{\mathrm{RHom}}_k(F,G)$ to $\underline{\mathrm{RHom}}_R(F \otimes_k R, G \otimes_k R)$. This assertion is standard.
\end{proof}

Recall that the functor $R\Gamma(A,-)$ carries $D^b_c(A,R)$ to $D_{perf}(R)$, and, as $A$ is compact, intertwines Verdier duality on $D^b_c(A,R)$ with the trivial duality $D_R(-) := \mathrm{RHom}_R(-,R)$ on $D_{perf}(R)$. Combining this with Lemma~\ref{lem:DualityTautological}, we obtain the following commutative diagram of functors:
\begin{equation}
\label{eq:DualityDiagram}
 \xymatrix{ D^b_c(A,k) \ar[rr]^-{\mathcal{L}_R \otimes_k (-)} \ar[d]_-{D_{A,k}(-)} & & D^b_c(A,R) \ar[d]^-{D_{A,R}(-)} \ar[rr]^-{R\Gamma(A,-)} && D_{perf}(R) \ar[d]^-{D_R(-)}  \\
		  D^b_c(A,k) \ar[rr]^-{\mathcal{L}^\vee_R \otimes_k (-)} && D^b(A,R) \ar[rr]^-{R\Gamma(A,-)} && D_{perf}(R). }
 \end{equation}
In particular, using Lemma~\ref{lem:DualTautological}, we arrive at the following compatibility:
\begin{equation}
\label{eq:DualityFM}
D_R\Big(R\Gamma(A, M \otimes_k \mathcal{L}_R)\Big) \simeq R\Gamma(A, D( [-1]^* M) \otimes_k \mathcal{L}_R).
\end{equation}
This allows us to define: 

\begin{definition}
For $M \in D^b_c(A,k)$, define its {\em Fourier transform} $\mathrm{FM}_A(M)$ as 
\[ \mathrm{FM}_A(M) := R\Gamma(A, M \otimes_k \mathcal{L}_R) \in D_{perf}(R).\]
\end{definition}

To justify the name, we show the following: for any $M \in D^b_c(A,k)$, the quasi-coherent complex $\mathrm{FM}_A(M)$ on $\mathrm{Char}(A)$ interpolates the cohomology of the character twists of $M$, as one would expect from a Fourier-type transform. More precisely, given a point $\chi \in \mathrm{Char}(A)$ with residue field $\kappa(\chi)$, we obtain an induced character $\tilde{\chi}:\pi_1(A) \xrightarrow{can} R^* \to \kappa(\chi)^*$, and thus an associated rank $1$ local system $L_\chi$ of $\kappa(\chi)$-modules on $A$. One then has:

\begin{lemma}
\label{lem:FibersFM}
For any $\chi \in \mathrm{Char}(A)$, there is a canonical isomorphism
\[ \mathrm{FM}_A(M) \otimes^L_R \kappa(\chi) \simeq R\Gamma(A, M \otimes_k L_\chi).\]
\end{lemma}
\begin{proof}
By the projection formula for $A \to \ast$, the left side identifies with 
\[ R\Gamma(A, M \otimes_k \mathcal{L}_R \otimes_R \kappa(\chi)).\]
As $\mathcal{L}_R$ is the $R$-local system associated to $\pi_1(A) \xrightarrow{can} R^*$, the base change $\mathcal{L}_R \otimes_R \kappa(\chi)$ is the $\kappa(\chi)$-local system associated to $\pi_1(A) \xrightarrow{can} R^* \to \kappa(\chi)^*$. But the latter is clearly also simply $L_\chi$, proving the claim.
\end{proof}

The key assertion responsible for Theorem~\ref{thm:GV} is:

\begin{proposition}
\label{prop:PerverseDualCoconnective}
The functor  $\mathrm{FM}_A(-)$ carries ${}^p D^{\geq 0}(A,k)$ into $D^{\geq 0}(R)$; the functor  $D_R(\mathrm{FM}_A(-))$ carries ${}^p D^{\leq 0}(A,k)$ into $D^{\geq 0}(R)$. In particular, if $M \in \mathrm{Perv}(A,k)$, then both $\mathrm{FM}_A(M)$ and $D_R(\mathrm{FM}_A(M))$ lie in $D^{\geq 0}(R)$. 
\end{proposition}
\begin{proof}
By equation~\eqref{eq:DualityFM}, it suffices to show $\mathrm{FM}_A(-)$ carries ${}^p D^{\geq 0}(A,k)$ into $D^{\geq 0}(R)$. For this, recall that $\mathcal{L}_R := \pi_! k$, where $\pi:V \to A$ is the universal cover. The projection formula gives 
\[ \mathrm{FM}_A(M) := R\Gamma(A, \mathcal{L}_R \otimes_k M) \simeq R\Gamma_c(V, \pi^* M)\]
 for any $M \in D^b_c(A,k)$. Now if $M \in { }^p D^{\geq 0}(A,k)$, then $\pi^*(M) \in {}^p D^{\geq 0}(V,k)$. Artin vanishing on the Stein space $V$ (see \cite[Theorem~10.3.8]{Kashiwara+Schapira:Sheaves}) implies that $R\Gamma_c(V,-)$ carries ${}^p D^{\geq 0}_c(V,k)$ into $D^{\geq 0}(k)$, proving the claim. 
\end{proof}

To pass from Proposition~\ref{prop:PerverseDualCoconnective} to the classical generic vanishing theorem, we recall the following (well-known) result in commutative algebra:

\begin{lemma}
\label{lem:CommAlgDuality}
Say $S$ is a noetherian ring with a dualizing complex $\omega_S^\bullet$, normalized so that the dualizing sheaf sits in cohomological degree $-\dim(S)$. Fix $M \in D^b_{coh}(S)$ and some integer $k$. Then the dual $D^\bullet_S(M) := \mathrm{RHom}_S(M, \omega_S^\bullet)$ lives in $D^{\geq -k}(S)$ if and only if $\dim(\mathrm{Supp}\ H^i(M)) \leq k-i$ for all $i$
\end{lemma}

This result can be found in \cite[Proposition 5.2]{Kashiwara:t-structures}, and a variant is implicit in \cite[Corollary IV.2.3]{Scholze:Torsion}.

\begin{proof}
For the forward implication, we recall the following fact about Grothendieck duality (see \cite[Tag 0A7U]{stacks-project}). If $N$ is a finitely generated $S$-module, then $\mathrm{Ext}_S^i(N,\omega_S^\bullet)$ is $0$ for $i  \notin \{-\dim(\mathrm{Supp}\ N),....,0\}$.  Now, for $M \in D^b_{coh}(S)$, consider the cohomological spectral sequence
\[ E_2^{i,j}: \mathrm{Ext}^i_S(H^{-j}(M), \omega_S^\bullet) \Rightarrow H^{i+j}(D_S^\bullet(M)).\]
As $M$ is bounded and $\omega_S^\bullet$ has finite injective dimension, there are no convergence problems. Now, if $\dim(H^i(M)) \leq k-i$ for all $i$, then the aforementioned fact shows that $E_2^{i,j} = 0$ if $i < -k - j$. The spectral sequence then shows $D_S^\bullet(M) \in D^{\geq -k}(S)$. 

For the converse, we need the following fact concerning the commutation of local duality with localization. If $(S,\mathfrak{m})$ is local and $\mathfrak{p} \subset S$ is a prime ideal of codimension $c_{\mathfrak{p}}$, we have $(\omega_S^\bullet)_{\mathfrak{p}} \simeq \omega_{S_{\mathfrak{p}}}^\bullet[c_{\mathfrak{p}}]$, and hence, for any $M \in D^b_{coh}(S)$, we get $D_S^\bullet(M)_{\mathfrak{p}} = D_{S_{\mathfrak{p}}}^\bullet(M_{\mathfrak{p}})[c_{\mathfrak{p}}]$. Now assume that $D_S^\bullet(M) \in D^{\geq -k}(S)$. We must show that $\dim(\mathrm{Supp}\ H^i(M)) \leq k-i$. For this, we may assume $S$ is local with maximal ideal $\mathfrak{m}$, and that the statement is known for all nontrivial localizations of $S$. Fix a nonmaximal prime $\mathfrak{p} \subset S$ of codimension $c_{\mathfrak{p}}$. Then our hypothesis gives $D_{S_{\mathfrak{p}}}^\bullet(M_{\mathfrak{p}})[c_{\mathfrak{p}}] \in D^{\geq -k}$, and hence $D_{S_{\mathfrak{p}}}^\bullet(M_{\mathfrak{p}}) \in D^{\geq -k + c_{\mathfrak{p}}}$. By induction and exactness of localization, we learn that $\dim(\mathrm{Supp}\ H^i(M)_{\mathfrak{p}}) \leq k-c_{\mathfrak{p}} - i$ for any such $\mathfrak{p}$. In particular, if $U = \mathrm{Spec}(S) - \{\mathfrak{m}\}$, then $\mathrm{Supp}\ H^i(M) \cap U$ has dimension $\leq k-1-i$ since $c_{\mathfrak{p}} \geq 1$ for any $\mathfrak{p} \in U$. As $\mathrm{Spec}(S)$ is obtained from $U$ by adding a single closed point that every point in $U$ specializes to, it follows immediately that $\mathrm{Supp}\ H^i(M)$ has dimension $\leq k-i$.
\end{proof}

\begin{remark}
\label{rmk:CommAlgDuality}
In the situation of Lemma~\ref{lem:CommAlgDuality}, it is sometimes convenient to work with dualizing complexes normalized slightly differently. Thus, set 
\[ D_S(M) = \mathrm{RHom}_S(M,\omega_S^\bullet[-\dim(S)]).\]
If $S$ is Gorenstein, this reduces to the trivial duality functor $\mathrm{RHom}(-,S)$ up to a twist. Lemma~\ref{lem:CommAlgDuality},  reinterpreted for $D_S$ instead of $D_S^\bullet$ and with $k = \dim(S)$, states: for $M \in D^b_{coh}(S)$, one has $D_S(M) \in D^{\geq 0}(S)$ if and only if $\mathrm{codim}(\mathrm{Supp}\ H^i(M)) \geq i$ for all $i$.
\end{remark}

One can now prove the generic vanishing theorem readily:

\begin{corollary}
If $M$ is a perverse sheaf on $A$, then one has the following:
\begin{enumerate}
\item $\mathrm{codim}(\mathrm{Supp}\ H^i(\mathrm{FM}_A(M))) \geq i$ and $\mathrm{codim}(\mathrm{Supp}\ H^i(D_R(\mathrm{FM}_A(M)))) \geq i$ for all $i$;
\item $H^i(A, M \otimes_k L_\chi) = 0$ for all $i \neq 0$ for $\chi$ in a non-empty Zariski open subset of $\mathrm{Char}(A)$;
\item $\chi(A,M) \geq 0$;
\item $\chi(A,M) = 0$ if and only if $R\Gamma(A, M \otimes L_\chi) = 0$ for some $\chi \in \mathrm{Char}(A)$.
\end{enumerate}
\end{corollary}

\begin{proof}
(1) follows from Proposition~\ref{prop:PerverseDualCoconnective} and Remark~\ref{rmk:CommAlgDuality}. As a consequence of (1), there is a nonempty Zariski open $U \subset \mathrm{Char}(A)$ such that $\mathrm{FM}_A(M)|_U$ is a locally free $\mathcal{O}_U$-module placed in degree $0$. On the other hand, if $\chi \in \mathrm{Char}(A)$, then Lemma~\ref{lem:FibersFM} shows that $\mathrm{FM}_A(M) \otimes^L_R \kappa(\chi) \simeq R\Gamma(A, M \otimes_R L_\chi)$. Now, if $\chi \in U$, then $\mathrm{FM}_A(M) \otimes^L_R \kappa(\chi) \simeq (\mathrm{FM}_A(M))|_U \otimes_{\mathcal{O}_U} \kappa(\chi)$ is concentrated in degree $0$ by our choice of $U$; thus, $R\Gamma(A, M \otimes_k L_\chi)$ is also concentrated in degree $0$, giving (2). Now (3) is immediate as $\chi(A,M) = \chi(A,M \otimes_k L_\chi)$ for any $\chi \in \mathrm{Char}(A)$ as they are both the Euler characteristics of different fibers of the perfect complex $\mathrm{FM}_A(M)$ on the connected variety $\mathrm{Char}(A)$. The same argument also proves $\Leftarrow$ in (4). Conversely, the implication $\Rightarrow$ in (4) comes from (2).
\end{proof}

\begin{remark}
It seems natural to ask if the results discussed in this section continue to hold for abelian varieties in positive characteristic, with $k$ being a finite ring whose order is invertible on the base. We do not know the answer to this question. The fundamental question seems to be the following: given an abelian variety $A$ over an algebraically closed field of characteristic $p$, a prime $\ell$ different from $p$, and a constructible sheaf $M$ of $\mathbf{F}_\ell$-vector spaces on $A$, does the direct limit
\[ \varinjlim_n H^i(A, [\ell^n]^* M) \]
vanish for $i > \dim(\mathrm{Supp}\ M)$? In other words, if $A_\infty$ denotes the inverse limit of the tower 
\[ \cdots \to A \xrightarrow{\ell} A \xrightarrow{\ell} A\]
of multiplication by $\ell$ maps on $A$, is the analog of Artin vanishing true for $A_\infty$? While we do not know the answer to this question, note that \cite{WeissauerAbVarFiniteField} does affirmatively answer the variant of this question where $A$ lives over $\overline{\mathbf{F}_p}$ and $M$ is a $\mathbf{Q}_\ell$-sheaf of geometric origin.
\end{remark}

\section{Codimension inequalities via Hard Lefschetz}
\label{sec:codimviaHL}

In this section, we make stronger hypothesis: Let $A$ be an abelian variety of dimension $g$ over $\mathbf{C}$, and assume that $k$ is a field of characteristic $0$. Let $R = k[\pi_1(A)]$, let $X = \mathrm{Spec}(R)$, and let $\mathrm{FM}_A:D^b_c(A,k) \to D_{perf}(R)$ be the Fourier transform from \S \ref{sec:FMTori}. Recall that for any $K \in D_{perf}(R)$, one has the cohomology support loci
\[ S^i(K) := \{x \in X \mid H^i(K \otimes_R \kappa(x)) \neq 0\}.\]
The subsets $S^i(K) \subset X$ are closed, and we set $S^i(A,M) := S^i(\mathrm{FM}_A(M))$ for any $M \in D^b_c(A,k)$; thus, the $k$-points of $S^i(A,M)$ coincide with the set of characters $\chi:\pi_1(A) \to k^*$ such that $H^i(A, M \otimes_k L_{\chi}) \neq 0$. Our goal is to prove the following estimate on the dimension of these subspaces:

\begin{theorem}
\label{thm:codimviaHL}
Fix $M \in D^b_c(A,k)$. Then we have:
\begin{enumerate}
\item If $M \in {}^p D^{\leq 0}(A,k)$, then $\mathrm{codim}(S^i(A,M)) \geq 2i$ for all $i \in \mathbf{Z}$.
\item If $M \in {}^p D^{\geq 0}(A,k)$, then $\mathrm{codim}(S^i(A,M)) \geq -2i$ for all $i \in \mathbf{Z}$.
\end{enumerate}
 In particular, if $M \in \mathrm{Perv}(A,k)$, then $\mathrm{codim}(S^i(A,M)) \geq |2i|$ for all $i$.
\end{theorem}

Our strategy is to prove Theorem~\ref{thm:codimviaHL} by exploiting the Hard Lefschetz theorem on $A$. Recall the statement:


\begin{theorem}[Hard Lefschetz]
\label{thm:HL}
If $c \in H^2(A,k)$ is the Chern class of an ample line bundle (ignoring twists) and if $N \in \mathrm{Perv}(A,k)$ is semisimple, then the cup product map 
\[ H^{-i}(A,N) \xrightarrow{c^i} H^i(A,N)\] 
is an isomorphism for any $i$.
\end{theorem}

\begin{remark}
When $N$ is of geometric origin, Theorem~\ref{thm:HL} follows from the theory of
mixed perverse sheaves \cite{BBD} (which builds on Deligne's \cite{DeligneWeil2}, and
works over a base field of any characteristic) or the work of Saito \cite{SaitoHodge}
on mixed Hodge modules. The general case was conjectured by Kashiwara
\cite{KashiwaraSemisimple}; in fact, he conjectured the same for any (i.e., not
necessarily regular) simple holonomic $\mathcal{D}$-module. For simple perverse
sheaves, this conjecture was proven using a specialization argument by Drinfeld
\cite{DrinfeldKashiwara} relying crucially on the work of Lafforgue
\cite{LafforgueGLn} and assuming a finiteness conjecture of de Jong
\cite{deJongFundamentalGroup} on the monodromy of lisse sheaves on varieties over
finite fields; the latter was proven independently by Gaitsgory
\cite{GaitsgorydeJong} and B\"ockle-Khare \cite{BockleKharedeJong}. An alternate
analytic proof was given by Sabbah \cite{SabbahTwistor} and Mochizuki for semisimple
local systems. The general case of simple holonomic $\mathcal{D}$-modules was
settled in a series of works by Mochizuki \cite{MochizukiWild}.
\end{remark}

For a simple perverse sheaf $M$, Theorem~\ref{thm:HL} implies a non-trivial statement about the fibers of $\mathrm{FM}_A(M)$. To ``integrate'' this fibral information over $X$, we use the following construction:

\begin{proposition}
\label{prop:PullbackFF}
The functor $D(R) \to D(A,k)$ defined by $N \mapsto \mathcal{L}_N := \mathcal{L}_R \otimes_R N$ satisfies the following:
\begin{enumerate}
\item It is left-adjoint to $M \mapsto R\Gamma(V, \pi^* M)$, where $\pi:V \to A$ is the universal cover, and the $R$-module structure on $R\Gamma(V, \pi^* M)$ is induced by the $\pi_1(A)$-equivariance of $\pi$.
\item It is fully faithful.
\end{enumerate}
\end{proposition}

Recall that $\mathcal{L}_R$ is an $R$-local system of rank $1$, obtained by descending the constant $R$-local system $R \in D(V,k)$ along the $\pi_1(A)$-torsor $\pi:V \to A$ using the tautological $\pi_1(A)$-action on $R$. Thus, one may view the complex $\mathcal{L}_N$ as the locally constant sheaf on $A$ whose pullback to $V$ is identified with the constant sheaf $N$, equipped with its canonical $\pi_1(A)$-equivariant structure coming from the $R$-module action. Unraveling definitions, one sees that $\mathcal{L}_{\kappa(x)} \simeq L_x$ for any point $x \in \mathrm{Char}(A)$.

\begin{proof}
For (1), given $N \in D(R)$ and $M \in D(A,k)$, we must check that
\[ \mathrm{RHom}_{D(A,k)} (\mathcal{L}_N, M) \simeq \mathrm{RHom}_R(N, R\Gamma(V, \pi^* M)).\]
By taking a free resolution for $N$, and observing that both sides behave similarly with respect to the free resolution, we may assume $N = R$. We must thus check that
\[ \mathrm{RHom}_{D(A,k)}(\mathcal{L}_R, M) \simeq R\Gamma(V, \pi^* M).\]
As $\mathcal{L}_R := \pi_! k$, and because $\pi^* \simeq \pi^!$, the left side simplifies to $\mathrm{RHom}_{D(V,k)}(k, \pi^* M) \simeq R\Gamma(V, \pi^* M)$, as wanted.

For (2), fix $N,N' \in D(R)$. We must check that the functor $N \mapsto \mathcal{L}_N$ induces an identification
\[ \mathrm{RHom}_R(N,N') \simeq \mathrm{RHom}_{D(A,k)}(\mathcal{L}_N, \mathcal{L}_{N'}).\]
By (1), the right side simplifies to $\mathrm{RHom}_R(N, R\Gamma(V, \pi^* \mathcal{L}_{N'}))$. Now $\pi^* \mathcal{L}_{N'}$ is the constant sheaf with value $N'$ (since the same is true when $N' = R$, by proper base change along $\pi$). As $V$ is contractible, this gives $R\Gamma(V, \pi^* \mathcal{L}_{N'}) \simeq N'$, which gives the desired identification.
\end{proof}

The full faithfulness above yields the following criterion for certain cup product maps to be $0$.

\begin{lemma}
\label{lem:HLVanishingResidueField}
Fix some $t \in H^i(A,k)$ and $x \in X=\mathrm{Spec}(k[\pi_1(A)])$. If $\mathrm{codim}(x) < i$, then the cup product with $t$ map 
\[ \cup\  t:\mathcal{L}_{\kappa(x)} \to \mathcal{L}_{\kappa(x)}[i]\]
is the $0$ map.
\end{lemma}
\begin{proof}
By Proposition~\ref{prop:PullbackFF}, the group $\mathrm{Hom}_{D(A,k)}(\mathcal{L}_{\kappa(x)}, \mathcal{L}_{\kappa(x)}[i])$ identifies with $\mathrm{Ext}^i_R(\kappa(x),\kappa(x))$. The latter can also be calculated as $\mathrm{Ext}^i_{R_x}(\kappa(x),\kappa(x))$, and hence vanishes if $i > \mathrm{codim}(x)$ since $R_x$ is a regular local ring of dimension $\mathrm{codim}(x) < i$.
\end{proof}

Exploiting the tension between a cup product map being an isomorphism (as in Theorem~\ref{thm:HL}) and $0$ (as in Lemma~\ref{lem:HLVanishingResidueField}), we can prove the main theorem of this section:

\begin{theorem}
\label{thm:CodimHLSupport}
Fix $M \in \mathrm{Perv}(A,k)$. Then $\mathrm{codim}(\mathrm{Supp}\ H^i(\mathrm{FM}_A(M))) \geq 2i$ for all $i \geq 0$. 
\end{theorem}
\begin{proof}
We may assume that $M$ is simple since the conclusion behaves well under exact sequences. For simple $M$,  we work by descending induction on $i$. The claim is clearly true for $i \gg 0$ as $\mathrm{FM}_A(M)$ is bounded. Fix an integer $i > 0$ and a point $x \in X$ of codimension $< 2i$. We must show that $H^i (\mathrm{FM}_A(M))_x = 0$. Induction lets us assume that $H^k(\mathrm{FM}_A(M))_x = 0$ for $k > i$. This implies that $H^i(\mathrm{FM}_A(M) \otimes_R \kappa(x)) \simeq H^i(\mathrm{FM}_A(M))_x/\mathfrak{m}_x$, where $\mathfrak{m}_x \subset R_x$ is the maximal ideal. By Nakayama, it thus suffices to check that $H^i(\mathrm{FM}_A(M) \otimes_R \kappa(x)) = 0$. Thanks to Lemma~\ref{lem:FibersFM}, this is equivalent to checking that $H^i(A, M \otimes_k \mathcal{L}_{\kappa(x)}) = 0$. Fix a class $c \in H^2(A,k)$ corresponding to the Chern class of an ample line bundle. The Hard Lefschetz theorem (see Theorem~\ref{thm:HL}) implies that the cup product with $c^i$ map
\[ \alpha:H^{-i}(A, M \otimes_k \mathcal{L}_{\kappa(x)}) \xrightarrow{c^i} H^{i}(A, M \otimes_k \mathcal{L}_{\kappa(x)})\]
is an isomorphism; here we implicitly use that $M \otimes_k \mathcal{L}_{\kappa(x)}$ is a semisimple perverse sheaf over $\kappa(x)$.\footnote{To see this, one can assume that $k$ is algebraically closed, in which case $M\otimes_k \mathcal L_{\kappa(x)}$ is actually simple, as one sees for example by using the classification of simple perverse sheaves as intermediate extensions of local systems.} This map is induced by applying the functor $H^0(A, M \otimes_k -)$ to the cup product with $c^i$ map
\[ \mathcal{L}_{\kappa(x)}[-i] \xrightarrow{c^i} \mathcal{L}_{\kappa(x)}[i].\]
Since $x$ has codimension $< 2i$, Lemma~\ref{lem:HLVanishingResidueField} tells us that this last map is $0$, and hence so is the map labelled $\alpha$ above. Thus, $\alpha$ is both the $0$ map and an isomorphism, so $H^i(A, M \otimes_k \mathcal{L}_{\kappa(x)}) = 0$, as wanted.
\end{proof}

\begin{proof}[Proof of Theorem~\ref{thm:codimviaHL}]
We first show that if $M \in \mathrm{Perv}(A,k)$, then $\mathrm{codim}(S^i(A,M)) \geq 2i$ for all $i$. A point $x \in X$ lies in $S^i(A,M)$ exactly when $H^i(\mathrm{FM}_A(M)(x)) \neq 0$. If $\mathrm{codim}(x) < 2i$, then Theorem~\ref{thm:CodimHLSupport} implies that $\mathrm{FM}_A(M)_x \in D^{< i}(R_x)$, and hence $H^i(\mathrm{FM}_A(M)(x)) = 0$;  this shows that any $x \in S^i(A,M)$ has $\mathrm{codim}(x) \geq 2i$.

The claim in (1) follows formally from the previous paragraph by expressing $M$ as an iterated extension of shifted perverse sheaves. For (2), observe that for any $x \in X$, the complex $R\Gamma(A, M \otimes_k L_x)$ is the $\kappa(x)$-linear dual of $R\Gamma(A, D_{A,k}(M) \otimes_k L_{x^{-1}})$: this results from the commutation of Verdier duality with $R\Gamma(A,-)$ and the formula $D_{A,\kappa(x)}(M \otimes_k L_x) = D_{A,k}(M) \otimes_k L_{x^{-1}}$. Hence, we have an equality 
\[ S^i(A,M) = \mathrm{inv}^* S^{-i}(A, D_{A,k}(M))\]
 as subspaces of $X$ (where $\mathrm{inv}:X \to X$ denotes inversion on the torus $X$), which immediately yields (2) from (1).
\end{proof}

\begin{remark}
It seems natural to wonder if Theorem~\ref{thm:codimviaHL} continues to hold for abelian varieties over a field of positive characteristic. More precisely, given an abelian variety $A$ over an algebraically closed field of characteristic $p$ and a prime $\ell$ different from $p$, one may define the ``open unit disc'' version $\widehat{\mathrm{Char}(A)}$ (as a rigid space over $\mathbf{Q}_\ell$) of the character variety $\mathrm{Char}(A)$, together with a Fourier transform functor $\widehat{\mathrm{FM}_A}:D^b_c(A, \mathbf{Q}_\ell) \to D^b_{coh}(\widehat{\mathrm{Char}(A)})$. One may then ask if Theorem~\ref{thm:CodimHLSupport} holds true for $\widehat{\mathrm{FM}_A}(M)$. For perverse sheaves $M$ of geometric origin, the main obstacle is proving the Hard Lefschetz theorem for the sheaves $M \otimes \mathcal{L}_{\chi}$ for varying characters $\chi$. We do not know how to prove this result; note that Drinfeld's conjecture $\mathrm{Kash}_\ell(k)$ from \cite[\S 1.7]{DrinfeldKashiwara} predicts a positive answer to a much more general version of this question.
\end{remark}

\section{Linearity}

Fix an abelian variety $A$ over $\mathbf{C}$, and let $k$ be a field of characteristic $0$. We follow the notation of \S \ref{sec:FMTori} above. Our goal is to show that the Fourier transform of a perverse sheaf has good linearity properties near the origin of $\mathrm{Char}(A)$, so let $S\simeq \widehat{\mathrm{Sym}}(H_1(A,k))$ be the completed local ring of $\mathrm{Char}(A)$ at the origin. The main theorem of this section is:

\begin{theorem}
\label{thm:Linearity}
Let $M \in \mathrm{Perv}(A,k)$ be a simple perverse sheaf of geometric origin. Then the completed stalk at $1\in \mathrm{Char}(A)$ of $\mathrm{FM}_A(M) \in D^b_{coh}(\mathrm{Char}(A))$ is represented by the $S$-complex
\[ \cdots \to H^{i-1}(A,M) \otimes_k S \to H^i(A,M) \otimes_k S \to H^{i+1}(A,M) \otimes_k S \to \cdots, \]
where the differential arises from the natural map $H^i(A,M) \to H^{i+1}(A,M) \otimes_k H_1(A,k)$ that is adjoint to the cup product  $H^1(A,k) \otimes_k H^i(A,M) \to H^{i+1}(A,M)$.
\end{theorem}

Note that one can apply the theorem also to twists of $M$ by torsion characters, so a similar result holds for the completion of $\mathrm{FM}_A(M)$ at torsion points in $\mathrm{Char}(A)$.

The strategy for proving the theorem is as follows. As the perverse sheaf $M$ is of geometric origin, we may specialize to the algebraic closure of a finite field. Now, by a version of the BGG correspondence, we reduce the linearity assertion above to a formality statement for the action of $R\Gamma(A,k)$ on $R\Gamma(A,M)$. This formality is then deduced from Deligne's theory \cite{DeligneWeil2} of weights (as cast in \cite{BBD}).

We recall the relevant version of the BGG correspondence in \S \ref{ss:BGG} using the language of $\infty$-categories \cite{LurieHTT} and higher algebra \cite{LurieHA}. The payoff for bringing in these tools is the material in \S \ref{ss:weights}: we give a quick deduction of some rather strong formality results that follow almost immediately from the theory of weights using higher algebra. With this ingredients in place, the strategy outlined in the previous paragraph is implemented in \S \ref{ss:LinearityCharp} in characteristic $p$, and the characteristic $0$ case follows by spreading-out.

\subsection{Recollections on the symmetric and exterior algebra duality}
\label{ss:BGG}

We begin with reminders on the Koszul duality relating exterior and symmetric algebras. The main difference, when compared to most standard standard  references, is that we do not restrict to the graded setting. All derived categories appearing in this section are viewed as stable $k$-linear $\infty$-categories (or, equivalently, differential graded (dg) categories over $k$, up to quasi-equivalence) in the sense of \cite[Chapter 1]{LurieHA}. The basic objects of interest are defined next:

\begin{notation}
Fix a field $k$ and a $k$-vector space $V$ of dimension $d$ with dual $W$. Let $S_0 = \mathrm{Sym}^*(V)$, and let $S$ be the completion of $S_0$ at the augmentation $S_0 \to k$ given by $V \mapsto 0$. Let $E := \mathrm{RHom}_S(k,k)$, viewed as an $E_1$-$k$-algebra or, equivalently, as an $E_1$-algebra in the symmetric monoidal category $D(k)$\footnote{For any commutative ring $A$, the notion of an $E_1$-$A$-algebra is one formalization of a homotopy-theoretically robust notion of an ``associative algebra in $D(A)$''; this notion is equivalent to that of either $A_\infty$-$A$-algebras or differential graded $A$-algebra (or, for short, $A$-dga), see \cite[Proposition 7.1.4.6]{LurieHA}.}. Then $V \simeq \mathfrak{m}/\mathfrak{m}^2$ is the cotangent space of $S$, and $H^*(E)$ is an exterior algebra on $W$. Moreover, attached to $E$, one has the derived $\infty$-category $D(E)$ with its distinguished object $k \in D(E)$. Let $D_{coh}(E) \subset D(E)$ be the full $\infty$-subcategory of those $M \in D(E)$ that have finite dimensional total homology (as $k$-modules); this is also smallest full stable $\infty$-subcategory that contains $k$, see Lemma~\ref{lem:DCohEGenerator} below. We will occassionally use the the natural $\mathbf{G}_m$-action\footnote{One can work equivalently with $\mathbf{Z}$-graded objects. However, to avoid confusing this grading with the cohomological degree, we stick to the $\mathbf{G}_m$-action perspective.} on $S_0$ (which gives $\mathrm{Sym}^n(V)$ weight $n$); this induces a $\mathbf{G}_m$-action on $E$ giving $H^i(E) \simeq \wedge^i W$ weight $-i$. We write a superscript of $\mathbf{G}_m$ on a derived $\infty$-category to denote the $\mathbf{G}_m$-equivariant version of the derived $\infty$-category; for example $D(k)^{\mathbf{G}_m}$ describes the derived $\infty$-category of graded $k$-vector spaces, and the $\mathbf{G}_m$-action on $V$ naturally lifts $E$ to an object of $D(k)^{\mathbf{G}_m}$.
\end{notation}

The following notion of formality for $E_1$-$k$-algebras will play an important role in the sequel.

\begin{definition}
\label{def:Formality}
An $E_1$-$k$-algebra $A$ is {\em formal} if there exists an isomorphism $\alpha_A:A \simeq H^*(A)$ of $E_1$-$k$-algebras inducing the identity on cohomology; here $H^*(A)$ is viewed as a differential graded $k$-algebra with trivial differential and $H^i(A)$ living in cohomological degree $i$. 
\end{definition}

We begin by observing that $E$ is formal in the best possible way:

\begin{lemma}
\label{lem:ExtAlgebraFormal}
The $E_1$-$k$-algebra $E$ is canonically $\mathbf{G}_m$-equivariantly formal, i.e., there exists a unique (up to contractible ambiguity) $\mathbf{G}_m$-equivariant isomorphism $E \stackrel{\alpha}{\simeq} H^*(E)$ of $E_1$-$k$-algebras in $D(k)^{\mathbf{G}_m}$ inducing the identity on cohomology.

\end{lemma}
\begin{proof}
This follows from exactly the same argument as the one used in Proposition~\ref{prop:PureComplexes}. Indeed, consider the full $\infty$-subcategory $D_{w=0}$ of $D(k)^{\mathbf{G}_m}$ spanned by complexes $K$ with $H^i(K)$ have weight exactly $-i$. The symmetric monoidal structure on $D(k)^{\mathbf{G}_m})$ induces one on $D_{w=0}$, and our definition of the $\mathbf{G}_m$-action on $V$ show that $E$ naturally lifts to an $E_1$-algebra in $D_{w=0}$. But every object in $D_{w=0}$ is canonically a direct sum of its cohomology groups (for weight reasons), so all $E_1$-algebras in $D_{w=0}$ are formal.
\end{proof}

For the rest of this section, we fix the (essentially unique) formality isomorphism $\alpha:E \simeq H^*(E) \simeq \wedge^* W$.

\begin{remark}
Let $(R,\mathfrak{m})$ be any regular local $k$-algebra with residue field $k$. Then we have $\mathrm{RHom}_R(k,k) \simeq \mathrm{RHom}_{\widehat{R}}(k,k)$ as $E_1$-$R$-algebras, where $\widehat{R}$ is the $\mathfrak{m}$-adic completion. Since $\widehat{R} \simeq S$ (non-canonically), we learn that $\mathrm{RHom}_R(k,k)$ is also formal by Lemma~\ref{lem:ExtAlgebraFormal}. 
\end{remark}

\begin{remark}
\label{rmk:GrModtoDerived}
Let $M^*$ be a graded $\wedge^* W$-module. We may view $M^*$ as a complex with trivial differential by placing $M^i$ in degree $i$; this complex is naturally a dg-module over the $k$-dga $\wedge^* W$. Let $\mathbf{G}_m$ act on this complex by giving the term $M^i$ weight $-i$. The resulting complex is a $\mathbf{G}_m$-equivariant dg-module over the $k$-dga $\wedge^* W$. Via the formality isomorphism $\alpha:E \simeq \wedge^* W$, this construction gives a functor
\[ \mathrm{Mod}^{gr}(\wedge^* W) := \{ \mathrm{graded} \wedge^* W\mathrm{-modules}\} \xrightarrow{\Phi} D(E)^{\mathbf{G}_m}. \]
If one restricts to finitely generated graded modules on the left side, one obtains a functor
\[  \mathrm{Mod}^{gr}_{fg}(\wedge^* W) \xrightarrow{\Phi} D_{coh}(E)^{\mathbf{G}_m}.\]
This functor will play a crucial role in this section.
\end{remark}

In order to do calculations in $D_{coh}(E)$, such as those of $\mathrm{Hom}$-sets, we need an effectively computable projective resolution of the generator $k$. A standard such resolution arises from a variant of the Koszul complex, recorded next:

\begin{construction}
\label{cons:ResolutionE}
The following quasi-isomorphism of $E$-complexes gives a $\mathbf{G}_m$-equivariant resolution of $k$:
\begin{equation}
\label{eq:ResolutionE} \Big(.... \to \Gamma^n(W)[-n] \otimes_k E \to  ...\to \Gamma^2(W)[-2] \otimes_k E \to W[-1] \otimes_k E \to E\Big) \xrightarrow{\simeq} k.
\end{equation}
Here $\Gamma^n(W) \simeq (\mathrm{Sym}^n(V))^\vee$ has weight $-n$. To describe the differential, note that 
\begin{equation}
 \begin{split}
 \mathrm{Hom}^{\mathbf{G}_m}_E(\Gamma^n(W)[-n] \otimes_k E, \Gamma^{n-1}[-n](W) \otimes_k E) & \simeq \mathrm{Hom}^{\mathbf{G}_m}_k(\Gamma^n(W)[-1], \Gamma^{n-1}(W) \otimes_k E) \\
& \simeq \mathrm{Hom}^{\mathbf{G}_m}_k(\mathrm{Sym}^n(V)^\vee, \mathrm{Sym}^{n-1}(V)^\vee \otimes_k E[1]) \\
& \simeq \mathrm{Hom}_k(\mathrm{Sym}^n(V)^\vee, \mathrm{Sym}^{n-1}(V)^\vee \otimes_k V^\vee).
\end{split}
\end{equation}
Under this isomorphism, the differential corresponds to the obvious map $\mathrm{Sym}^n(V)^\vee \to \mathrm{Sym}^{n-1}(V)^\vee \otimes V^\vee$, dual to multiplication in the symmetric algebra. The reason this defines a complex is that the canonical map $\mathrm{Sym}^{n-1}(V) \otimes V \otimes V \to \mathrm{Sym}^{n+1}(V)$ vanishes on the subspace of $V \otimes V$ that is dual to the quotient $W \otimes W \to \wedge^2 W$. 
\end{construction}

Using this resolution, we check that $k$ generates $D_{coh}(E)$, as promised earlier.

\begin{lemma}
\label{lem:DCohEGenerator}
The category $D_{coh}(E)$ is the smallest stable $\infty$-subcategory of $D(E)$ that contains $k$.
\end{lemma}
\begin{proof}
Using the resolution from Construction~\ref{cons:ResolutionE}, one checks that if $M \in D(E)$ lies in $D^{\leq n}(k)$, then any map $E[-n] \to M$ factors through the canonical map $E[-n] \to k[-n]$. More precisely, the complexes $\mathrm{RHom}_E(k[-n],M)$ and $\mathrm{RHom}_E(E[-n],M)$ are connective, and the canonical map from first one to the second one is surjective on $H^0$. In particular, if $M \in D_{coh}(E)$, then one can find a map $k[-n] \to M$ in $D_{coh}(E)$ whose cone $Q$ has smaller dimensional total homology than $M$ as a $k$-vector space. Proceeding this way, it follows that $D_{coh}(E)$ is generated by $k$ under finite direct sums, shifts, and cones.
\end{proof}

The next lemma relates perfect complexes on $S$ to coherent complexes on $E$; the idea here is roughly that perfect complexes on $S$ can be computed via {\em derived} Cech descent along the map $\mathrm{Spec}(k) \to \mathrm{Spec}(S)$. This lemma may also be viewed as a variant of the BGG construction, the main difference being that we do not restrict to the graded setting. 

\begin{proposition}
\label{prop:BGGNonGraded}
The functor $F := \mathrm{RHom}_S(k,-)$ gives an equivalence $D_{perf}(S) \to D_{coh}(E)$ with inverse $G := \mathrm{RHom}_E(k,-)[d] \otimes_k \wedge^d V$.
\end{proposition}
\begin{proof}
The $\infty$-category $D_{perf}(S)$ of perfect $S$-complexes is the smallest stable $\infty$-subcategory of $D(S)$ that contains $S$ and is closed under retracts (by definition of perfect complexes). By Lemma~\ref{lem:DCohEGenerator}, the $\infty$-category $D_{coh}(E)$ of coherent $E$-complexes is the smallest stable $\infty$-subcategory of $D(E)$ that contains $k$. Moreover, this $\infty$-category is automatically closed under retracts in $D(E)$ by the definition of $D_{coh}(E)$. 
Also, we have $F(S) = \mathrm{RHom}_S(k,S) \simeq k[-d] \otimes_k (\wedge^d V)^\vee$ by a Koszul cohomology calculation; as the latter module is identified with $k[-d]$ (after fixing a trivialization of $\wedge^d V$), the functor $F$ carries a generator to a generator. It thus suffices to check that $F$ induces an isomorphism $\mathrm{RHom}_S(S,S) \simeq \mathrm{RHom}_E(k[-d],k[-d])$ (and similarly for $G$, which actually reduces to the same calculation). The left side is just $S$, so we must calculate that the right side is also $S$. In other words, we need to check that the natural map $S \to \mathrm{RHom}_E(k,k)$ is an isomorphism. This is a standard calculation (see the equivalences $E$ and $C$ in \cite[Theorem 2.1]{DwyerGreenleesCompleteTorsion} for a much more general statement), and we briefly sketch how it works. Using the the resolution in Construction~\ref{cons:ResolutionE} to resolve the first copy of $k$ (and representing $E$ as the Koszul complex over $S$ on $W$ to ensure we have $S$-free modules everywhere), we learn that $\mathrm{RHom}_E(k,k)$ is computed by the product totalization of a fourth quadrant bicomplex whose $n$-th column is the (standard Koszul resolution over $S$ for) $k$-complex $\mathrm{Sym}^n(V)[n]$, and the map $\mathrm{Sym}^n(V)[n] \to \mathrm{Sym}^{n+1}(V)[n+1]$ induced by the horizontal differential going from the $n$-th column to the $(n+1)$-st column classifies the standard $S$-module extension of $\mathrm{Sym}^n(V)$ by $\mathrm{Sym}^{n+1}(V)$. Collapsing to a single complex then gives the desired identification $\mathrm{RHom}_E(k,k) \simeq S$.
\end{proof}

\begin{remark}
\label{rmk:BGGGradedFunctor}
The equivalences $F$ and $G$ are $\mathbf{G}_m$-equivariant, and they are $\mathbf{G}_m$-equivariantly inverse to one another. We write 
\[ \xymatrix{ D_{coh}(E)^{\mathbf{G}_m} \ar@<0.5ex>[r]^-{G^{\mathbf{G}_m}} &  D_{perf}(S)^{\mathbf{G}_m} \ar@<0.5ex>[l]^-{F^{\mathbf{G}_m}} }\]
for the induced equivalences at the level of equivariant derived categories. With this notation, the constructions thus far in this section may be summarized in the following commutative diagram:
\begin{equation}
\label{diag:BGG}
\xymatrix{ \mathrm{Mod}^{gr}_{fg}(\wedge^* W) \ar[r]^-{\Phi} \ar[rd]^-{\Psi} & D_{coh}(E)^{\mathbf{G}_m} \ar@<0.5ex>[rr]^{G^{\mathbf{G}_m}} \ar[d]^-{\mathrm{forget}} & &D_{perf}(S)^{\mathbf{G}_m} \ar[d]^{\mathrm{forget}} \ar@<0.5ex>[ll]^-{F^{\mathbf{G}_m}} \\
		& D_{coh}(E) \ar@<0.5ex>[rr]^G & & D_{perf}(S) \ar@<0.5ex>[ll]^-{F} }
\end{equation}

Here the functors $F^{\mathbf{G}_m}$ and $G^{\mathbf{G}_m}$ are defined as Remark~\ref{rmk:BGGGradedFunctor}, the functor $\Phi$ is defined in Remark~\ref{rmk:GrModtoDerived}, and $\Psi$ is defined as the composition making the triangle commute. 
\end{remark}

The explicit identification of the composite $G \circ \Psi$ (or, better, $G^{\mathbf{G}_m} \circ \Phi$) in the diagram above forms the basis of the BGG correspondence \cite{BGG} relating graded $H^*(E)$-modules to linear $S$-complexes, and is recalled next.

\begin{construction}[BGG]
\label{cons:BGG}
Let $M^*$ be a graded $\wedge^* W$-module. Assume that $M^i$ is a finite dimensional $k$-vector space for each $i$, and $M^i = 0$ for $|i| \gg 0$. By Remark~\ref{rmk:GrModtoDerived} and our finiteness assumption, this gives $M := \Phi(M^*) \in D_{coh}(E)^{\mathbf{G}_m}$. Our goal is to construct a canonical identification in $D_{perf}(S)^{\mathbf{G}_m}$ of the form
\[ G^{\mathbf{G}_m}(M) \simeq \Big( ... \to M^0 \otimes_k S \to M^1 \otimes_k S \to .... \to M^n \otimes_k S \to ...\Big)[d] \otimes_k \wedge^d V, \]
where $M^i$ has weight $-i$, and the map $M^i \otimes_k S \to M^{i+1} \otimes_k S$ is $\mathbf{G}_m$-equivariant $S$-module map defined by the canonical map $M^i \to M^{i+1} \otimes_k V \to M^{i+1} \otimes_k S$ coming from the action map $W \otimes M^i \to M^{i+1}$. In particular, since $M^i \otimes_k S \simeq S(i)^{\oplus \dim(M^i)}$, the complex $G^{\mathbf{G}_m}(M)$ is a {\em linear complex}, i.e. it can be represented by a complex of graded free $S$-modules where the term in cohomological degree $i$ is isomorphic to $S(c + i)^{\oplus n_i}$ for suitable $n_i \geq 0$ and constant $c$ independent of $i$.

Since $G^{\mathbf{G}_m}(-) := \mathrm{RHom}_E(k,-)[d] \otimes_k \wedge^d V$, it suffices to calculate  $\mathrm{RHom}_E(k,M)$. The $\mathbf{G}_m$-equivariant resolution from Construction~\ref{cons:ResolutionE} shows that $\mathrm{RHom}_E(k,M)$ is the product totalization of the $\mathbf{G}_m$-equivariant bicomplex
\[ K := M^* \to M^* \otimes V[1] \to M^* \otimes \mathrm{Sym}^2(V)[2] \to ... \to M^* \otimes \mathrm{Sym}^n(V)[n] \to .... ,\]
with $K^{p,q} = M^{p+q} \otimes_k \mathrm{Sym}^p(V)$, trivial vertical differentials (since $M^*$ has a trivial differential), and horizontal differentials $M^{p+q} \otimes_k \mathrm{Sym}^p(V) \to M^{p+q+1} \otimes_k \mathrm{Sym}^{p+1}(V)$ from the action of $W$ on $M^*$. Taking the product totalization, we learn that $\mathrm{RHom}_E(k,M)$ is calculated by the following $\mathbf{G}_m$-equivariant perfect complex over $S$
\[ ... \to M^0 \otimes_k S \to M^1 \otimes_k S \to .... \to M^n \otimes_k S \to ..., \]
as asserted. 
\end{construction}

\subsection{Weights and pure complexes}
\label{ss:weights}

We recall the basic structure of the category in which the $\ell$-adic cohomology of smooth projective varieties over finite fields takes values.

\begin{notation}
\label{not:Weights}
Fix a prime $p$, a power $q = p^r$, and a prime $\ell$ different from $p$. Let $k = \overline{\mathbf{Q}_\ell}$. For $i \in \mathbf{R}$, we say that $\alpha \in k^*$ is a {\em Weil number of weight $i$} if $|\tau(\alpha)| = q^{\frac{i}{2}}$ for every embedding $\tau:k \hookrightarrow \mathbf{C}$. Let $W^i \subset k^*$ denote the set of all possible Weil numbers of weight $i$. Note that $W^i \cap W^j = \emptyset$ for $i \neq j$, and $W^i \cdot W^j \subset W^{i+j}$. For any variety $X_0$ over $\mathbf{F}_q$ with base change $X$ to $\overline{\mathbf{F}_q}$, the complex $R\Gamma(X, k)$ acquires a natural action of the ($q$-power) geometric Frobenius. We view the resulting Galois representation $R\Gamma(X,k)$ as an object of the $\infty$-category $D_{perf}(k[t,t^{-1}])$ having finite length cohomology sheaves;  here $t$ acts by geometric Frobenius. By \cite{DeligneWeil2}, any such complex is supported on $\sqcup_i W^i \subset \mathrm{Spec}(k[t,t^{-1}])$.
\end{notation}

We isolate the main subcategory of interest: 

\begin{definition}
A complex $K \in D_{perf}(k[t,t^{-1}])$ is {\em pure of weight $0$} if $H^i(K)$ is supported set-theoretically on $W^i \subset \mathrm{Spec}(k[t,t^{-1}])$ for each $i$. The collection of all such complexes spans a full $\infty$-subcategory $D_{pure,0} \subset D_{perf}(k[t,t^{-1}])$. 
\end{definition}

By perfectness, any $K \in D_{pure,0}$ has finite length when viewed as a $k[t,t^{-1}]$-complex, i.e., the complex has finitely many non-zero cohomology modules, and each of those modules is a finite length $k[t,t^{-1}]$-module. The key observation is that pure complexes inhabit a discrete world. In fact, they are all equivalent to their cohomology groups in an essentially unique way.

\begin{proposition}
\label{prop:PureComplexes}
For any $K \in D_{pure,0}$, there is a canonical isomorphism $\alpha_K:K \simeq \oplus_i H^i(K)[-i]$. Moreover, $D_{pure,0} \subset D(k[t,t^{-1}])$ is a discrete subcategory, i.e., $\mathrm{Ext}^i_{k[t,t^{-1}]}(K,L) = 0$ for $i < 0$ and $K,L \in D_{pure,0}$.
\end{proposition}
\begin{proof}
For any finite length $K \in D_{perf}(k[t,t^{-1}])$, we have a canonical map
\[ \alpha_K:K \to \prod_i \prod_{x \in W^i} K_x \simeq \bigoplus_i \bigoplus_{x \in W^i} K_x.\]
This map is an isomorphism if $K$ is acylic outside $\sqcup_i W^i \subset \mathrm{Spec}(k[t,t^{-1}])$ (and thus for $K \in D_{pure,0}$). Moreover, if $K \in D_{pure,0}$, then the right side identifies with $\oplus_i H^i(K)[-i]$ by defintion, giving the desired identification. For the second part, the same argument shows that
\[ \mathrm{RHom}_{k[t,t^{-1}]}(K,L) \simeq \prod_i \mathrm{RHom}_{k[t,t^{-1}]}(H^i(K),H^i(L)).\]
The right side clearly has no cohomology in negative degrees, proving the claim.
\end{proof}

The fundamental examples of pure complexes come from geometry:

\begin{example}
\label{ex:Purity}
Let $X_0$ be a proper variety over $\mathbf{F}_q$, and let $M_0$ be a pure perverse sheaf of weight $0$ on $M_0$. If $X$ and $M$ denote the base change of $X_0$ and $M_0$ to $\overline{\mathbf{F}_q}$, then $R\Gamma(X,M)$ is pure of weight $0$. When $X_0$ is smooth and $M$ is lisse, this comes from \cite[Corollary 3.3.6]{DeligneWeil2}; the general case comes from \cite[Corollary 5.4.2]{BBD}.
\end{example}

Note that the category $D_{pure,0}$ inherits a natural symmetric monoidal structure:

\begin{construction}
As $W^i \cdot W^j \subset W^{i+j}$, the $\infty$-category $D_{pure,0}$ is endowed with a symmetric monoidal structure $\otimes$ given by the tensor product of the underlying complexes of $k$-vector spaces, with $t$ acting diagonally; under the identification of $\mathrm{Spec}(k[t,t^{-1}])$ with the group $\mathbf{G}_m$, this corresponds to convolution. The object $R\Gamma(X,k)\in D_{pure,0}$ as in Example~\ref{ex:Purity} is an $E_\infty$-algebra for this symmetric monoidal structure.
\end{construction}

Algebraic structures in $D_{pure,0}$ defined using this symmetric monoidal structure are formal:

\begin{corollary}
\label{cor:PurityFormal}
All $E_1$-algebras and their modules in $D_{pure,0}$ are canonically formal. 
\end{corollary}
\begin{proof}
This follows immediately from Proposition~\ref{prop:PureComplexes} as the isomorphism $K \simeq \oplus_i H^i(K)[-i]$ in that proposition is compatible with the symmetric monoidal structure.
\end{proof}

\begin{remark}
Corollary~\ref{cor:PurityFormal} applies equally well to $E_\infty$-algebras, showing that the $E_\infty$-algebra $R\Gamma(X,k)$ is formal for a smooth proper variety $X/\overline{\mathbf{F}_q}$. Via the proper and smooth base change isomorphisms, it follows that the same assertion holds true for smooth proper varieties $X/\mathbf{C}$. In particular, after fixing an isomorphism $k \simeq \mathbf{C}$ and using Artin's comparison theorem $R\Gamma(X,k) \simeq R\Gamma_{sing}(X^{an}, \mathbf{C})$, this reproves the formality result from \cite{DGMS} for algebraic varieties. In fact, the proof above essentially fleshes out a heuristic argument outlined in  \cite[page 246, paragraph 1]{DGMS} using the modern language of $\infty$-categories.
\end{remark}

\subsection{Abelian varieties over finite fields}
\label{ss:LinearityCharp}

In this section, we use the theory of weights to establish a linearity result for the stalk at the origin of the Fourier transform of a simple perverse sheaf of geometric origin. In fact, working exclusively over finite fields, we are able to prove a stronger result thanks to the work of Lafforgue.

\begin{notation}
\label{not:LinearityCharp}
Let $A_0$ be an abelian variety of dimension $g$ over a finite field $\mathbf{F}_q$, and write $A$ for its base change to $\overline{\mathbf{F}_q}$. Set $k = \overline{\mathbf{Q}_\ell}$. Write $E := R\Gamma(A,k)$, viewed as an $E_1$-$k$-algebra. Let $S = k \llbracket \pi_1(A)_\ell \rrbracket$ be the completed group algebra\footnote{More precisely, we define $S$ as the completion of $S_0 := \mathbf{Z}_\ell \llbracket \pi_1(A)_{\ell} \rrbracket \otimes_{\mathbf{Z}_\ell} k$ at the point $S_0 \to k$ corresponding to the trivial character, where $\mathbf{Z}_\ell \llbracket \pi_1(A)_{\ell} \rrbracket$ is the Iwasawa algebra of the pro-$\ell$-fundamental group of $A$, and is identified non-canonically with a power series ring in $2g$ variables over $\mathbf{Z}_\ell$. In particular, $S$ is a power series in the same set of variables over $k$.} of the pro-$\ell$ part of the geometric fundamental group $\pi_1(A)$ of $A$, so $S$ is the completion of $\mathrm{Sym}^*(H_1(A,k))$ at its augmentation ideal; we then have a canonical identification $E \simeq \mathrm{RHom}_S(k,k)$ arising from the $K(\pi,1)$-nature of abelian varieties. In particular, we may use the notation of \S \ref{ss:BGG} with $V = H_1(A,k)$, and $W = H^1(A,k)$. There is a tautological character $\pi_1(A) \to S^*$, which defines an $S$-local system $\mathcal{L}_S$ on $A$. For any $M \in D^b_c(A,k)$, write $\widehat{\mathrm{FM}}_A(M) := R\Gamma(A, M \otimes_k \mathcal{L}_S) \in D_{perf}(S)$. Note that there is a $\mathrm{Gal}(\overline{\mathbf{F}_q}/\mathbf{F}_q)$-action on $S$, $E$, and also on $\widehat{\mathrm{FM}}_A(M)$ if $M$ is defined over $\mathbf{F}_q$.
\end{notation}

We can describe the stalk at $0$ of the Fourier transform of certain sheaves explicitly:

\begin{theorem}
\label{thm:LinearityCharp}
Let $M_0 \in \mathrm{Perv}(A_0,k)$ be absolutely irreducible, and let $M$ be its base change to $\overline{\mathbf{F}_q}$. Then there is a natural $\mathrm{Gal}(\overline{\mathbf{F}_q}/\mathbf{F}_q)$-equivariant identification
\[ \widehat{\mathrm{FM}}_A(M) \simeq \Big(... \to H^i(A,M) \otimes_k S \to H^{i+1}(A,M) \otimes_k S \to ... \Big) \in D_{perf}(S), \]
where the differentials on the right are defined by the natural map $H^i(A,M) \to H^{i+1}(A,M) \otimes_k H_1(A,k)$ coming from the cup product action $H^1(A,k) \otimes H^i(A,M) \to H^{i+1}(A,M)$.
\end{theorem}
\begin{proof}
First, we observe that there is a canonical formality isomorphism $\mathrm{can}:E \simeq H^*(A,k)$ by purity via Corollary~\ref{cor:PurityFormal}. This allows us to contemplate the following diagram:
\[  \xymatrix{ \mathrm{Mod}^{gr}_{fg}(H^*(E))_{pure,0} \ar[rrr]^-{\mathrm{forget\ Frobenius}} \ar[dd]^-{\Phi} &  & & \mathrm{Mod}^{gr}_{fg}(H^*(E)) \ar[d]^-{\Phi} & & \\
		& & & D_{coh}(E)^{\mathbf{G}_m} \ar@<0.5ex>[rr]^{G^{\mathbf{G}_m}} \ar[d]^-{\mathrm{forget}} & &D_{perf}(S)^{\mathbf{G}_m} \ar[d]^{\mathrm{forget}} \ar@<0.5ex>[ll]^-{F^{\mathbf{G}_m}} \\
D_{coh}(E)_{pure,0} \ar[rrr]^-{\mathrm{forget\ Frobenius}} &	&	& D_{coh}(E) \ar@<0.5ex>[rr]^G & & D_{perf}(S) \ar@<0.5ex>[ll]^-{F} }
\]
Here the second and third columns come from diagram \eqref{diag:BGG}; the category $D_{coh}(E)_{pure,0}$ denotes coherent $E$-complexes in $D_{pure,0}$ (i.e., the underlying complex with Frobenius action lies in $D_{pure,0}$), and similarly for $\mathrm{Mod}^{gr}_{fg}(H^*(E))_{pure,0}$; the horizontal maps labelled `forget Frobenius'  are obtained by forgetting the $\mathrm{Gal}(\overline{\mathbf{F}_q}/\mathbf{F}_q)$-action. Crucially, by Corollary~\ref{cor:PurityFormal}, the leftmost vertical map $\Phi$ is an equivalence with inverse given by taking cohomology.

To prove the theorem, we are allowed to replace $M_0$ with a twist by rank $1$ local system on $\mathrm{Spec}(\mathbf{F}_q)$. By Lafforgue \cite[Corollary VII.8]{LafforgueGLn}, we may thus assume that $M_0$ is pure of weight $0$. Now consider $\widehat{\mathrm{FM}}_A(M) \in D_{perf}(S)$. This is a $\mathrm{Gal}(\overline{\mathbf{F}_q}/\mathbf{F}_q)$-equivariant complex, so its image under $F$ is $\mathrm{Gal}(\overline{\mathbf{F}_q}/\mathbf{F}_q)$-equivariant as well. Moreover, by construction, this image is canonically identified with
\begin{equation}
\begin{split}
F(\widehat{\mathrm{FM}}_A(M)) & := \mathrm{RHom}_S(k, \widehat{\mathrm{FM}}_A(M)) \\
& \simeq \big(\mathrm{FM}_A(M) \otimes_S k\big)[-2g] \otimes_k (\wedge^{2g} H_1(A,k))^\vee \\
& \simeq R\Gamma(A,M) \otimes_k (\wedge^{2g} H_1(A,k))^\vee[-2g],
\end{split}
\end{equation}
where we used the Galois equivariant isomorphism $\mathrm{RHom}_S(k,-) \simeq (\wedge^{2g} H_1(A,k))^\vee \otimes_S k \otimes_S (-)$ of functors, as in Proposition~\ref{prop:BGGNonGraded}, to arrive at the second isomorphism above. As this isomorphism is Galois equivariant, it follows from purity of $M_0$ (see Example~\ref{ex:Purity}) that $F(\widehat{\mathrm{FM}}_A(M))$ is pure of weight $0$ (note that $(\wedge^{2g} H_1(A,k))^\vee[-2g]$ is pure of weight $0$), and hence comes from $D_{coh}(E)_{pure,0}$. As the left most vertical map is an equivalence with inverse given by taking cohomology, the preceding diagram shows that $\widehat{\mathrm{FM}}_A(M) \in D_{perf}(S)$ comes from $H^*(A,M) \in \mathrm{Mod}^{gr}_{fg}(H^*(E))_{pure,0}$ in the diagram above. The result now follows from the explicit identification of the functor $G^{\mathbf{G}_m} \circ \Phi$ given in Construction~\ref{cons:BGG}.
\end{proof} 
%
%

In particular, we get Theorem~\ref{thm:Linearity} using the material developed above.

\begin{proof}[Proof of Theorem~\ref{thm:Linearity}]
Since $M$ has geometric origin, one can use spreading out arguments to find an abelian variety over a finite field so that the base change $\mathrm{FM}_A(M) \otimes_R S$ can be calculated using $\ell$-adic cohomology as in Notation~\ref{not:LinearityCharp}. Then the result follows from Theorem~\ref{thm:LinearityCharp}.
\end{proof}

\def\cprime{$'$}

\end{document}